\documentclass[final,3p,times,twocolumn]{elsarticle}
\DeclareOldFontCommand{\sc}{\normalfont\scshape}{\@nomath\sc}
\usepackage{amsmath,amssymb,amsfonts,amsthm,mathrsfs,comment,mathdots,multirow,tikz}
\usepackage{array}
\usepackage{blkarray}
\usepackage{capt-of}
\usepackage{color}
\usepackage{colortbl}
\usepackage{fancyhdr}
\usepackage{fancyvrb}
\usepackage{float}
\usepackage[T1]{fontenc}
\usepackage{framed}
\usepackage{geometry}
\usepackage{graphicx}
\usepackage[latin1]{inputenc}
\usepackage{stmaryrd}
\usepackage{upquote}
\usepackage{url}
\usepackage{xspace}	
\usepackage{booktabs}
\usepackage{mathtools}
\usepackage{soul}
\DeclarePairedDelimiter{\ceil}{\lceil}{\rceil}
\definecolor{lightgray}{gray}{0.5}
\definecolor{gray}{gray}{0.7}
\definecolor{darkgreen}{rgb}{0,0.6,0.13}
\newcommand{\nc}{\newcommand}
\nc{\dsp}{\displaystyle}
\nc{\txt}{\textstyle}
\nc{\reff}[1]{(\ref{#1})}
\nc{\mrm}[1]{\mathrm{#1}}
\nc{\udl}[1]{\underline{#1}}
\nc{\ovl}[1]{\overline{#1}}
\nc{\al}{\underline{\boldsymbol{\alpha}}}
\nc{\la}{\underline{\boldsymbol{\lambda}}}
\nc{\llbr}{\llbracket}
\nc{\rrbr}{\rrbracket}
\nc{\lbr}{\lbrack}
\nc{\rbr}{\rbrack}
\nc{\N}{\mathbb{N}}
\nc{\Z}{\mathbb{Z}}
\nc{\D}{\mathbb{D}}
\nc{\Q}{\mathbb{Q}}
\nc{\R}{\mathbb{R}}
\nc{\C}{\mathbb{C}}
\nc{\T}{\mathbb{T}}
\nc{\Stwo}{\mathbb{S}^2}
\nc{\tld}[1]{\tilde{#1}}
\nc{\wtld}[1]{\widetilde{#1}}
\nc{\hu}{\hat{u}}
\nc{\wh}[1]{\widehat{#1}}
\nc{\Fbf}{\textbf{F}}
\nc{\Gbf}{\textbf{G}}
\nc{\Lbf}{\textbf{L}}
\nc{\Nbf}{\textbf{N}}
\nc{\Ibf}{\textbf{I}}
\nc{\Dbf}{\textbf{D}} 
\nc{\Tbf}{\textbf{T}}
\nc{\Jbf}{\textbf{J}} 
\nc{\Rbf}{\textbf{R}}   
\nc{\ph}{\varphi}
\nc{\NN}{\mathcal{NN}}
\nc{\OO}{\mathcal{O}}
\nc{\sumeven}{\sum_{k=-N/2}^{N/2}{\hspace{-0.3cm}}'{\;\,}}
\nc{\sumevenk}{\sum_{k=-n/2}^{n/2}{\hspace{-0.3cm}}'{\;\,}}
\nc{\sumevenj}{\sum_{j=-m/2}^{m/2}{\hspace{-0.3cm}}'{\;\,}}
\nc{\sumodd}{\sum_{k=-\frac{N-1}{2}}^{\frac{N-1}{2}}}
\nc{\sumoddl}{\sum_{l=-\frac{N-1}{2}}^{\frac{N-1}{2}}}
\nc{\cqfd}{~\hbox{\vrule width 2.5pt depth 2.5 pt height 3.5 pt}}
\nc{\rr}[1]{\textcolor{red}{#1}}
\nc{\bb}[1]{\textcolor{blue}{#1}}
\nc{\rrr}[1]{\textcolor{red}{\textit{\textbf{new:}}\,\,#1}}
\nc{\bbb}[1]{\textcolor{blue}{\textit{\textbf{old:}}\,\,#1}}
\nc{\hm}[1]{\textcolor{red}{\textbf{HM: \textit{#1}}}}
\nc{\hz}[1]{\textcolor{blue}{\textbf{HZ: \textit{#1}}}}
\nc{\ra}[1]{\renewcommand{\arraystretch}{#1}}
\nc{\bs}[1]{\boldsymbol{#1}}
\nc{\wt}{\widetilde{\times}}
\nc{\wtt}{\widetilde{\varprod}}
\nc{\wpp}{\widetilde{p}}
\nc{\wff}{\tilde{f}}
\nc{\wgg}{\widetilde{g}}
\nc{\wTT}{\widetilde{T}}
\nc{\wKK}{\widetilde{K}}
\nc{\wphi}{\widetilde{\phi}}
\nc{\wPhi}{\widetilde{\Phi}}
\nc{\wpsi}{\widetilde{\psi}}
\nc{\wPsi}{\widetilde{\Psi}}
\nc{\wom}{\widetilde{\omega}}
\newtheorem{theorem}{Theorem}[section]
\newtheorem{corollary}[theorem]{Corollary}
\newtheorem{proposition}[theorem]{Proposition}

\definecolor{parulaB}{RGB}{0,113.985,188.955}

\setulcolor{red}
\journal{Neural Networks}

\begin{document}

\begin{frontmatter}

\title{Error bounds for deep ReLU networks using the Kolmogorov--Arnold superposition theorem}

\author[HM]{Hadrien Montanelli\corref{cor1}}
\ead{hadrien.montanelli@gmail.com}

\author[HY]{Haizhao Yang}
\ead{haizhao@nus.edu.sg}

\cortext[cor1]{Corresponding author}
\address[HM]{Department of Applied Physics and Applied Mathematics, Columbia University, New York, United States}
\address[HY]{Department of Mathematics, National University of Singapore, Singapore}

\begin{abstract}
We prove a theorem concerning the approximation of multivariate functions by deep ReLU networks, for which the curse of the dimensionality is lessened. Our theorem is based on a constructive proof of the Kolmogorov--Arnold superposition theorem, and on a subset of multivariate continuous functions whose outer superposition functions can be efficiently approximated by deep ReLU networks.
\end{abstract}


\end{frontmatter}

%%%%%%%%%%%%%%%%%%%%%%%%%%%%%%%%%%%%%%%%%%%%%%%%%%%%%%%%%%%%%%%%%%%%%%%%%%%
\section{Introduction}

At the second International Congress of Mathematicians in Paris 1900, Hilbert presented ten of his 23 problems, including the 13th problem about equations of degree seven. He considered the following equation,
\begin{align*}
x^7 + ax^3 + bx^2 + cx + 1 = 0,
\end{align*}
and asked whether its solution $x(a,b,c)$, seen as a function of the three parameters $a$, $b$ and $c$, could be written as the composition of functions of only two variables.

Hilbert's 13th problem was solved by Kolmogorov and his 19 years old student Arnold in a series of papers in the 1950s. Kolmogorov first proved in 1956 that any continuous function of several variables could be expressed as the composition of functions of three variables \cite{kolmogorov1956}. His student Arnold extended his theorem in 1957; three variables were reduced to two \cite{arnold1957}. Kolmogorov finally showed later that year that functions of only one variable were needed \cite{kolmogorov1957}. The latter result is known as the \textit{Kolmogorov--Arnold superposition theorem}, and states that any continuous functions $f:[0,1]^n\rightarrow\R$ can be decomposed as
\begin{align*}
f(x_1,\ldots,x_n) = \sum_{j=0}^{2n}\phi_j\left(\sum_{i=1}^n\psi_{i,j}(x_i)\right),
\end{align*}
with $2n+1$ continuous \textit{outer} functions $\phi_j:\R\rightarrow\R$ (dependent of $f$) and $2n^2+n$ continuous \textit{inner} functions $\psi_{i,j}:[0,1]\rightarrow\R$ (independent of $f$). 

The Kolmogorov--Arnold superposition theorem was further improved in the 1960s and the 1970s. Lorentz showed in 1962 that the outer functions $\phi_j$ might be chosen to be the same function $\phi$, and replaced the inner functions $\psi_{i,j}$ by $\lambda_i\psi_j$, for some positive rationally independent constants $\lambda_i\leq 1$ \cite{lorentz1962}, while Sprecher replaced the inner functions $\psi_{i,j}$ by H\"{o}lder continuous functions $x_i\mapsto\lambda^{ij}\psi(x_i+j\epsilon)$ in 1965 \cite{sprecher1965}. Two years later, Fridman demonstrated that the inner functions could be chosen to be Lipschitz continuous, but his decomposition used $2n+1$ outer functions and $2n^2+n$ inner functions \cite{fridman1967}. Finally, Sprecher provided in 1972 a decomposition with Lipschitz continuous functions $x_i\mapsto\lambda^{i-1}\psi(x_i+j\epsilon)$ \cite{sprecher1972}.

Theoretical connections with neural networks started with the work of Hecht--Nielsen in 1987 \cite{hechtnielsen1987}. He interpreted the Kolmogorov--Arnold superposition theorem as a neural network, whose activation functions were the inner and outer functions. Girosi and Poggio claimed in 1989 that his interpretation was irrelevant for two reasons; first, the inner and outer functions were highly nonsmooth (\textit{i.e.}, these were at least as difficult to approximate as $f$); second, the outer functions depended on $f$ (\textit{i.e.}, the network architecture could not be parametrized). K{\r u}rkov\'a weakened the statement of Girosi and Poggio, in the early 1990s, by giving a direct proof of the universal approximation theorem of multilayer neural networks using the Kolmogorov--Arnold superposition theorem, and by showing that the weight selection reduced to a linear regression problem \cite{kurkova1991, kurkova1992}.

Numerical implementations originated with the work of Sprecher in the mid 1990s \cite{sprecher1996, sprecher1997}, which was followed, in 2003, by the Kolmogorov's spline network of Igelnik and Parikh \cite{igelnik2003}. Braun and Griebel proposed {an algorithm to implement a constructive proof of the Kolmogorov--Arnold theorem} in 2009 \cite{braun2009}, using K\"{o}ppen's H\"{o}lder continuous inner function \cite{koppen2002}. 

Approximation theory for neural networks started with shallow networks and the 1989 universal approximation theorems of Cybenko \cite{cybenko1989} and Hornik \cite{hornik1989}. In the last few years, the attention has shifted to the approximation properties of deep ReLU networks \cite{bach2017, cohen2016, eldan2016, montanelli2019a, montanelli2019b, petersen2018, poggio2017, shaham2018, telgarsky2016, yarotsky2017, yarotsky2018, shen2019}. In particular, one of the most important theoretical problems is to determine why and when deep networks lessen or break the \textit{curse of dimensionality}, characterized by the $\OO(\epsilon^{-n})$ growth of the network size $W$ as the error $\epsilon\rightarrow0$, in dimension $n$.\footnote{We recall that $W=\OO(\epsilon^{-n})$ means that there exists $c_1(n)>0$, such that $W\leq c_1(n)\epsilon^{-n}$, for sufficiently small values of $\epsilon$. Alternatively, we shall write $\epsilon=\OO(W^{-1/n})$ when there exists $c_2(n)>0$, such that $\epsilon\leq c_2(n)W^{-1/n}$, for sufficiently large values of $W$.} {We recommend the review} \cite{poggio2017} {for a discussion about the curse of dimensionality in the context of deep network approximation.}

In this paper, we introduce a set of multivariate continuous functions for which the approximation of the outer functions by deep ReLU networks is appealing to lessen the curse of the dimensionality. 
We show that any function $f:[0,1]^n\rightarrow\R$ in this set can be approximated with error $\epsilon$ by a very deep ReLU network of depth and size\footnote{{Following Yarotsky} \cite{yarotsky2017}{, we define the \textit{depth} $L$ of a network as the number of layers, the \textit{size} $W$ as the total number of weights, and we allow connections between units in non-neighboring layers.}} $\OO\left(\epsilon^{-\log n}\right)$; the curse of dimensionality is lessened.

Before the exposition of our main result in Section~\ref{sec:theorem}, we will review a specific version of the Kolmogorov--Arnold superposition theorem in Section~\ref{sec:kast}, and show in Section~\ref{sec:innerouter} how to approximate the inner and outer functions by very deep ReLU networks.

%%%%%%%%%%%%%%%%%%%%%%%%%%%%%%%%%%%%%%%%%%%%%%%%%%%%%%%%%%%%%%%%%%%%%%%%%%%
\section{Constructive version of the Kolmogorov--Arnold superposition theorem}\label{sec:kast}

We review in this section a constructive version of the Kolmogorov--Arnold superposition theorem that goes back to Sprecher in 1996 and 1997 \cite{sprecher1996, sprecher1997}. The proof he provided at the time was not fully correct; minor modifications were made by Braun and Griebel in 2009 to complete his proof \cite[Thm.~2.1]{braun2009}, using the inner function suggested by K\"{o}ppen \cite{koppen2002}.

For any integer $n\geq2$, $m\geq 2n$ and $\gamma\geq m+2$, let 
\begin{align}\label{eq:a}
a = \frac{1}{\gamma(\gamma-1)},
\end{align}
\begin{align}\label{eq:lambda}
\lambda_1=1, \quad \lambda_i = \sum_{\ell=1}^\infty\gamma^{-(i-1)\beta_n(\ell)}, \quad 2\leq i\leq n,
\end{align}
with 
\begin{align}
\beta_n(\ell) = \frac{1-n^\ell}{1-n} = 1 + n + \ldots + n^{\ell-1},
\end{align}
and
\begin{align}\label{eq:nualpha}
\nu = 2^{-\alpha}(\gamma+3), \quad \alpha = \log_\gamma2.
\end{align}

We recall that a function $f:[a,b]\rightarrow\R$ is said to be $(\nu,\alpha)$-H\"{o}lder continuous if and only if there exist scalars $\nu>0$ and $0<\alpha\leq1$, such that $\vert f(x)-f(y)\vert\leq\nu\vert x-y\vert^\alpha$, for all $x,y\in[a,b]$. (The value $\alpha=1$ yields $\nu$-Lipschitz continuous functions.)

\begin{theorem}[Kolmogorov--Arnold superposition theorem]\label{thm:kast}
Let $n\geq2$, $m\geq2n$ and $\gamma\geq m+2$ be given integers, and let $a$, $\lambda_i$ ($1\leq i\leq n$), $\nu$ and $\alpha$ be defined as in Equations~\eqref{eq:a}--\eqref{eq:nualpha}. Then, there exists a $(\nu,\alpha)$-H\"{o}lder continuous inner function $\psi:[0,2)\rightarrow[0,2)$, such that for any continuous function $f:[0,1]^n\rightarrow\R$, there exist $m+1$ continuous outer function $\phi_j:[0,2\frac{\gamma-1}{\gamma-2})\rightarrow\R$, such that
\begin{align}\label{eq:kast}
f(x_1,\ldots,x_n) = \sum_{j=0}^{m}\phi_j\left(\sum_{i=1}^n\lambda_i\psi(x_i+ja)\right).
\end{align}
\end{theorem}

%%%%%%%%%%%%%%%%%%%%%%%%%%%%%%%%%%%%%%%%%%%%%%%%%%%%%%%%%%%%%%%%%%%%%%%%%%%
Let us now go through the main two steps of the proof of Theorem~\ref{thm:kast}; for details, see \cite{sprecher1996, sprecher1997, braun2009}.

The first step is the building of the inner function $\psi$, which involves uniform grids $D_k$ with step sizes $\gamma^{-k}$,
\begin{align*}
D_k = \{i\gamma^{-k},\,0\leq i\leq\gamma^{k}-1\}\subset[0,1).
\end{align*}
There are $\gamma^k$ different points $0\leq d\leq1-\gamma^{-k}<1$ on each grid $D_k$, and each point $d$ on $D_k$ is represented in base $\gamma$ as follows,
\begin{align*}
d = \sum_{\ell=1}^k i_\ell\gamma^{-\ell}, \quad i_\ell\in\{0,1,\ldots,\gamma-1\}.
\end{align*}

%%%%%%%%%%%%%%%%%%%%%%%%%%%%%%%%%%%%%%%%%%%%%%%%%%%%%%%%%%%%%%%%%%%%%%%%%%%
\begin{figure}[t]
\centering
\includegraphics[scale=.35]{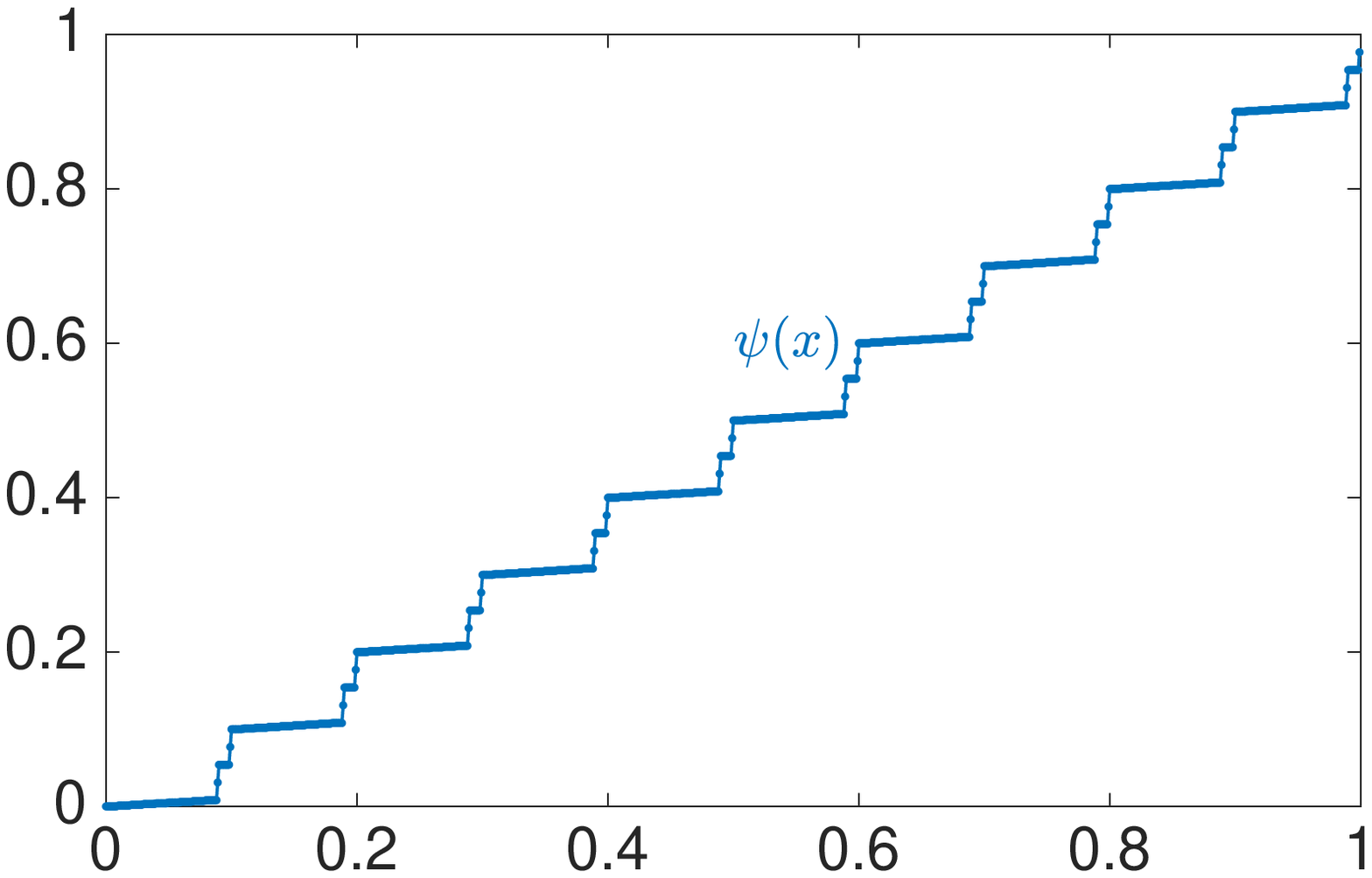} 
\includegraphics[scale=.35]{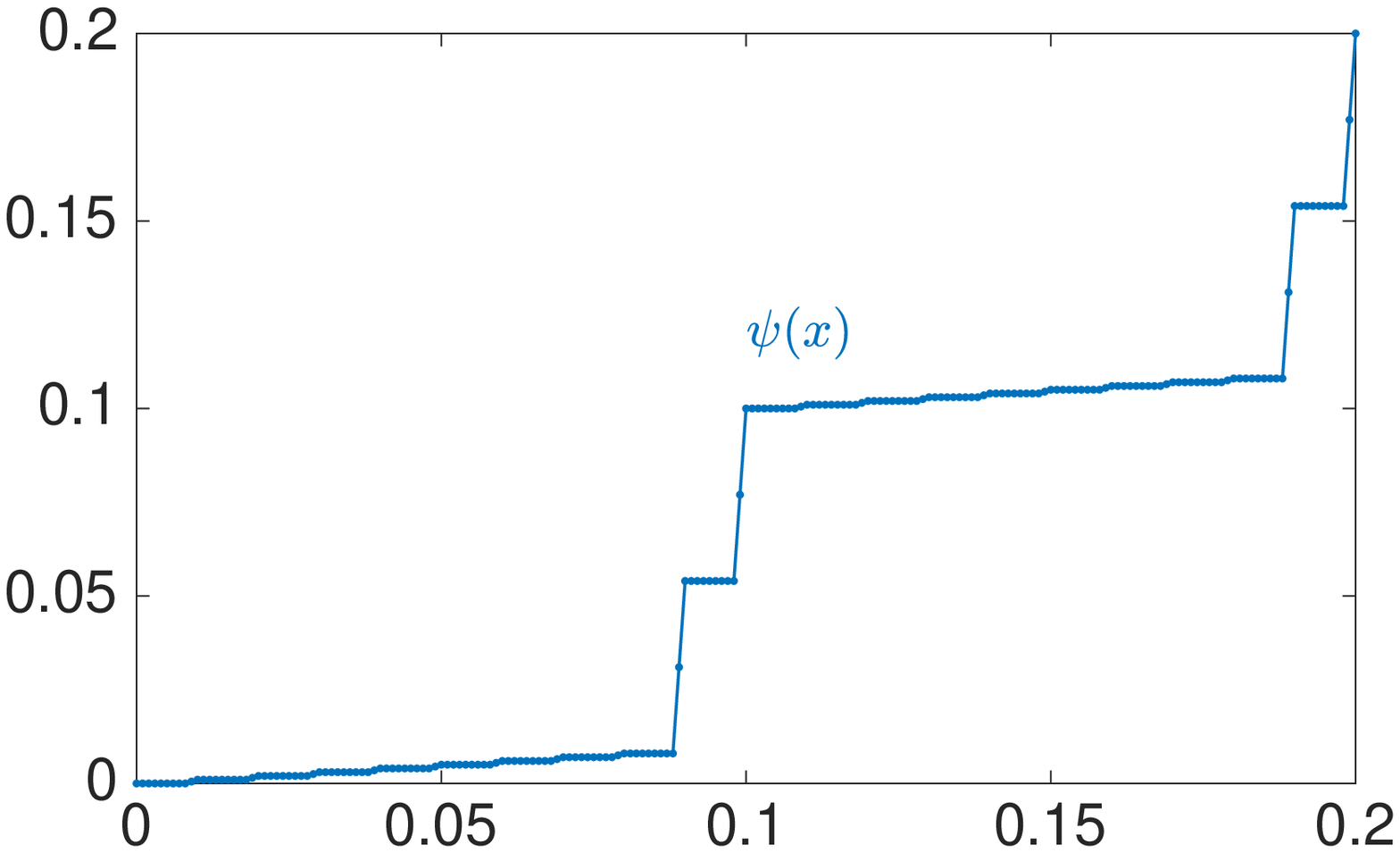} 
\caption{\textit{Plot of the inner function $\psi$ evaluated on the grid $D_3$ with $n=2$ and $\gamma=10$ (top). 
The second row is a zoomed plot that reveals the self-similarity of the graph of $\psi$ as $k\rightarrow\infty$.}}
\label{fig:psi}
\end{figure}

%%%%%%%%%%%%%%%%%%%%%%%%%%%%%%%%%%%%%%%%%%%%%%%%%%%%%%%%%%%%%%%%%%%%%%%%%%%
\begin{proposition}[Construction of the inner function]\label{prop:innercon}
The inner function $\psi$ is first defined at grid points $d\in D_k$ via $\psi(d)=\psi_k(d)$ for all integers $k\geq1$, where the functions $\psi_k$ are recursively defined by
$$
\psi_k(d) =
\left\{
\begin{array}{ll}
\hspace{-.15cm}d, & \hspace{-.15cm}d\in D_1, \\
\hspace{-.15cm}\psi_{k-1}\left(d-i_k\gamma^{-k}\right)\,+ \\
\quad\hspace{-.15cm}i_k\gamma^{-\beta_n(k)}, & \hspace{-.15cm}d\in D_k,\,k>1,\,i_k<\gamma-1, \\
\hspace{-.15cm}\frac{1}{2}\Big[\psi_k\left(d-\gamma^{-k}\right)\,+ \\
\quad\hspace{-.15cm}\psi_{k-1}\left(d+\gamma^{-k}\right)\Big], & \hspace{-.15cm}d\in D_k,\,k>1,\,i_k=\gamma-1.
\end{array}
\right.
$$

The function $\psi$ is then defined at any $x\in[0,1)$ via\footnote{{The existence of the limit is based on a suitably defined Cauchy sequence; see} \cite[Lem.~2.3]{braun2009} {for details.}}
\begin{align*}
\psi(x) = \lim_{k\rightarrow\infty}\psi_k\left(\sum_{\ell=1}^ki_\ell\gamma^{-\ell}\right),
\end{align*}
since each $x\in[0,1)$ has the representation
\begin{align*}
x = \sum_{\ell=1}^\infty i_\ell\gamma^{-\ell} = \lim_{k\rightarrow\infty}\sum_{\ell=1}^ki_\ell\gamma^{-\ell}.
\end{align*}
Finally, the inner function is extended to $x\in[1,2)$ by
\begin{align*}
\psi(x) = \psi(x-1) + 1.
\end{align*}
The resulting function has domain and range $[0,2)$.
\end{proposition}

For points $d=\sum_{\ell=1}^ki_\ell\gamma^{-\ell}\in D_k$ whose indices $i_\ell$ are all strictly smaller than $\gamma-1$, it is easy to show, by induction, that
\begin{align*}
\psi(d) = \sum_{\ell=1}^ki_\ell\gamma^{-\beta_n(\ell)}.
\end{align*}
For other points, the right-hand side in the equation above is only a lower bound.

The inner function constructed in Proposition~\ref{prop:innercon} was introduced by K\"oppen in 2002 \cite{koppen2002}. It is H\"{o}lder continuous, a result that can be proved using the techniques introduced by Sprecher in his 1965 paper \cite{sprecher1965}.

%%%%%%%%%%%%%%%%%%%%%%%%%%%%%%%%%%%%%%%%%%%%%%%%%%%%%%%%%%%%%%%%%%%%%%%%%%%
\begin{proposition}[H\"{o}lder continuity of the inner function]\label{prop:innerhold}
The inner function $\psi$ of Proposition~\ref{prop:innercon} is $(\nu,\alpha)$-H\"{o}lder continuous with $\nu=2^{-\alpha}(\gamma+3)$ and $\alpha=\log_\gamma2$.
\end{proposition}

\begin{proof}
See \cite[Sec.~4]{sprecher1965}.
\end{proof}

We plot in Figure~\ref{fig:psi} the graph of the function $\psi$ evaluated on the grid $D_3$ for $n=2$ and $\gamma=10$. As $k\rightarrow\infty$, the graph of $\psi$ exhibits self-similarity, which is expected since $\psi$ is merely H\"{o}lder continuous.

%%%%%%%%%%%%%%%%%%%%%%%%%%%%%%%%%%%%%%%%%%%%%%%%%%%%%%%%%%%%%%%%%%%%%%%%%%%
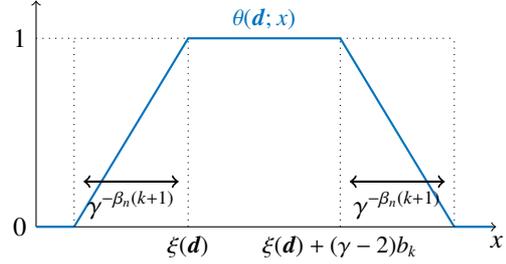
\begin{figure}[t]
\centering
\begin{tikzpicture}

% horizontal axis:
\draw[->] (-0.5,0) -- (5.55,0) node[anchor=north] {$\small{x}$};
\draw (1.5,0) node[anchor=north] {$\small{\xi(\bs{d})}$};
\draw (3.5,0) node[anchor=north] {$\small{\xi(\bs{d})+(\gamma-2)b_k}$};

% vertical axis:
\draw[->] (-0.5,0) -- (-0.5,3) node[anchor=east] {};
\draw (-0.5,0) node[anchor=east] {$\small{0}$};
\draw (-0.5,2.5) node[anchor=east] {$\small{1}$};

% function:
\draw[dotted] (-0.5,2.5) -- (1.5,2.5);
\draw[dotted] (3.5,2.5) -- (5,2.5);
\draw[dotted] (1.5,0) -- (1.5,2.5);
\draw[dotted] (3.5,0) -- (3.5,2.5);
\draw[thick,color=parulaB] (-0.5,0) -- (0,0) -- (1.5,2.5) -- (3.5,2.5) -- (5,0) -- (5.5,0);
\draw[color=parulaB] (2.5,2.75) node {$\small{\theta(\bs{d};x)}$};

% arrows:
\node[draw=none] (ghost1) at (0, 0.6) {};
\node[draw=none] (ghost2) at (1.5, 0.6) {};
\draw [thick, <->] (ghost1) -- node[below]{$\small{\gamma^{-\beta_n(k+1)}}$} (ghost2);
\draw[dotted] (0,0) -- (0,2.5);
\node[draw=none] (ghost3) at (3.5, 0.6) {};
\node[draw=none] (ghost4) at (5, 0.6) {};
\draw [thick, <->] (ghost3) -- node[below]{$\small{\gamma^{-\beta_n(k+1)}}$} (ghost4);
\draw[dotted] (5,0) -- (5,2.5);

\end{tikzpicture}
\caption{\textit{For each $\bs{d}\in (D_k^j)^n$, the function $\theta(\bs{d};\cdot)$ is compactly supported and piecewise linear with slope $\pm\gamma^{\beta_n(k+1)}$. Therefore, it is $\nu$-Lipschitz continuous with $\nu=\gamma^{\beta_n(k+1)}$.}}
\label{fig:theta}
\end{figure}

%%%%%%%%%%%%%%%%%%%%%%%%%%%%%%%%%%%%%%%%%%%%%%%%%%%%%%%%%%%%%%%%%%%%%%%%%%%
The second step of the proof is the iterative construction of the outer functions $\phi_j$. For each $0\leq j\leq m$, let $D^j_k$ denote the shifted grid defined by
\begin{align*}
D_k^j & = D_k + j\sum_{\ell=2}^k\gamma^{-\ell}, \quad 0\leq j\leq m.
\end{align*}
Let $(D_k)^n$ and $(D_k^j)^n$ denote the Cartesian products of $n$ copies of $D_k$ and $D_k^j$, and let
\begin{align*}
\xi(\bs{d}) = \sum_{i=1}^n\lambda_i\psi(d_i), \quad \bs{d}=(d_1,\ldots,d_n)\in(D_k^j)^n,
\end{align*}
and
\begin{align*}
b_k = \left(\sum_{\ell=k+1}^\infty\gamma^{-\beta_n(\ell)}\right)\left(\sum_{i=1}^n\lambda_i\right).
\end{align*}

Finally, for each $\bs{d}\in(D_k^j)^n$, let $\theta:x\mapsto\theta(\bs{d};x)$ denote the function defined by
\begin{align*}
& \theta(\bs{d};x) = \sigma\left(\gamma^{\beta_n(k+1)}\left[x - \xi(\bs{d})\right] + 1\right) \\
& \quad\quad\quad\quad - \sigma\left(\gamma^{\beta_n(k+1)}\left[x - \xi(\bs{d}) - (\gamma - 2)b_k\right]\right),
\end{align*}
where $\sigma:\R\rightarrow[0,1]$ is the piecewise linear function satisfying $\sigma(x)=0$ for $x\leq0$, $\sigma(x)=x$ for $0\leq x\leq1$, and $\sigma(x)=1$ for $x\geq1$. For given $k\geq 1$ and $0\leq j\leq m$, the $\gamma^{nk}$ functions $\theta(\bs{d};\cdot)$ have disjoint supports, and are $\nu$-Lipschitz with $\nu=\gamma^{\beta_n(k+1)}$; see Figure~\ref{fig:theta}.

%%%%%%%%%%%%%%%%%%%%%%%%%%%%%%%%%%%%%%%%%%%%%%%%%%%%%%%%%%%%%%%%%%%%%%%%%%%
\begin{proposition}[Construction of the outer functions]\label{prop:outercon}
Let $\delta$ and $\eta$ be two scalars that verify
\begin{align*}
0<\delta<1-\frac{n}{n-m+1}
\end{align*}
and
\begin{align*}
0<\frac{m-n+1}{n+1}\delta + \frac{2n}{m+1}\leq\eta<1,
\end{align*}
and $f:[0,1]^n\rightarrow\R$ be a continuous function.

Starting with $f_0=0$ and $e_0=f-f_0=f$, the approximate outer function $\phi_j^r$ at iteration $r\geq1$ are defined, for each $0\leq j\leq m$, as
\begin{align*}
\phi_j^r(x) = \frac{1}{m+1}\sum_{\ell=1}^r\sum_{\bs{d}\in (D_{k_\ell})^n} e_{\ell-1}(\bs{d})\theta\left(\bs{d}+j\sum_{i=2}^{k_\ell}\gamma^{-i};x\right),
\end{align*}
for some $k_r=k_r(f)$ chosen such that $\Vert\bs{x}-\bs{x}'\Vert_\infty\leq\gamma^{-k_r}$ implies $\vert e_{r-1}(\bs{x}) - e_{r-1}(\bs{x}')\vert\leq\delta\Vert e_{r-1}\Vert_{L^\infty([0,1]^n)}$.

This yields an approximate function $f_r$,
\begin{align}\label{eq:f_r}
f_r(x_1,\ldots,x_n) = \sum_{j=0}^m\phi_j^r\left(\sum_{i=1}^n\lambda_i\psi(x_i+ja)\right),
\end{align}
and its error $e_r = f - f_r$, with
\begin{align}\label{eq:e_r}
\Vert e_r\Vert_{L^\infty([0,1]^n)}\leq\eta^{r}\Vert f\Vert_{L^\infty([0,1]^n)}.
\end{align}

Taking the limit $r\rightarrow\infty$ yields
\begin{align*}
f(x_1,\ldots,x_n) = \sum_{j=0}^m\phi_j\left(\sum_{i=1}^n\lambda_i\psi(x_i+ja)\right),
\end{align*}
where $\phi_j = \lim_{r\rightarrow\infty}\phi_j^r$.\footnote{{The existence of the limits as $r\rightarrow\infty$ relies on $\phi_j^r$ being bounded and} Equation~\eqref{eq:e_r}{; see} \cite[Cor.~3.9]{braun2009} {for details.}}
\end{proposition}

The approximate outer functions $\phi_j^r$ of Proposition~\ref{prop:outercon} are Lipschitz continuous, as we shall prove next.

%%%%%%%%%%%%%%%%%%%%%%%%%%%%%%%%%%%%%%%%%%%%%%%%%%%%%%%%%%%%%%%%%%%%%%%%%%%
\begin{proposition}[Lipschitz continuity of the outer functions]\label{prop:outerlip}
For all $r\geq1$ and $0\leq j\leq m$, the outer functions $\phi_j^r$ of Proposition~\ref{prop:outercon} have domain $[0,2\frac{\gamma-1}{\gamma-2})$, and are $\nu_r(f)$-Lipschitz continuous with
\begin{align}\label{eq:nu_r}
\nu_r(f) = \frac{\Vert f\Vert_{L^\infty([0,1]^n)}}{m+1}\sum_{\ell=1}^r\eta^{\ell-1}\gamma^{\beta_n(k_\ell(f)+1)}.
\end{align}
\end{proposition}

\begin{proof}
To prove that the domain is $[0,2\frac{\gamma-1}{\gamma-2})$, we use the fact that $\vert\psi(x)\vert<2$ for all $x\in[0,2)$, and
\begin{align*}
\sum_{i=1}^n\lambda_i & < 1 + \frac{1}{\gamma-1} + \frac{1}{\gamma^{1+n}-1} + \frac{1}{\gamma^{1+n+n^2}-1} + \ldots, \\
& < \frac{\gamma-1}{\gamma-2}.
\end{align*}

For the Lipschitz constant, we recall that, for given $k_\ell(f)$ and $j$, the functions $x\mapsto\theta(\bs{d};x)$, $\bs{d}\in(D_{k_\ell}^j)^n$, have disjoint supports, and are $\nu(f)$-Lipschitz continuous with $\nu(f)=\gamma^{\beta_n(k_\ell(f)+1)}$. Using Equation~\eqref{eq:e_r}, summing over $\ell$ and multiplying by $1/(m+1)$ yields the desired result.
\end{proof}

Let us emphasize that the Lipschitz constants $\nu_r(f)$ in Proposition~\ref{prop:outerlip} depend on $f$ via the integers $k_\ell(f)$. 
This motivates us to introduce a set of continuous functions based on the growth of $k_\ell(f)$ with $\ell$ as follows,
\begin{equation*}
K_C([0,1]^n;\,\R) = \big\{f\in C([0,1]^n;\,\R),\,k_r(f)\leq C,\,r\geq1\big\},
\end{equation*}
for some constant $C>0$, where $C([0,1]^n;\,\R)$ denotes the set of multivariate continuous functions, and for given $n\geq 2$, $m\geq n$, $\gamma\geq m+2$, $\delta$ and $\eta$.
A direct calculation shows that functions in this set have outer functions whose Lipschitz constants \eqref{eq:nu_r} satisfy
\begin{align}\label{eq:nu_r_K_C}
\nu_r(f) \leq \frac{\Vert f\Vert_{L^\infty([0,1]^n)}}{m+1}r\gamma^{2n^C}.
\end{align}

%%%%%%%%%%%%%%%%%%%%%%%%%%%%%%%%%%%%%%%%%%%%%%%%%%%%%%%%%%%%%%%%%%%%%%%%%%%
\section{Approximation of the inner and outer functions by very deep ReLU networks}\label{sec:innerouter}

Let $\omega:[0,\infty)\rightarrow[0,\infty)$ be a function that is vanishing and continuous at $0$, \textit{i.e.}, 
$\lim_{\delta\rightarrow0+}\omega(\delta)=\omega(0)=0$, and $B\subset\R^d$ be a compact domain.
We say that an uniformly continuous function $f:B\rightarrow\R$ has modulus of continuity $\omega$ if and only if
\begin{align*}
\vert f(\bs{x})-f(\bs{x}')\vert\leq\omega(\Vert\bs{x}-\bs{x}'\Vert_2),\quad\forall\bs{x},\bs{x}'\in B.
\end{align*}

Many classical estimates in approximation theory are based on moduli of continuity. For example, best degree-$d$ polynomial approximation of continuous functions of one variable with modulus of continuity $\omega$ yields $\mathcal{O}(\omega(d^{-1}))$ errors \cite[Thm.~3.9]{gil2007}. The $\mathcal{O}(\omega(d^{-1/n}))$ errors in dimension $n$ suffers from the curse of dimensionality, but matches the lower bound obtained by nonlinear widths \cite[Thm.~4.2]{devore1989}.

In neural network approximation, moduli of continuity appear in the work of Yarotsky. In 2018, he proved that very deep ReLU networks of depth $L=\OO(W)$ and size $W$ generate $\mathcal{O}(\omega(\OO(W^{-2/n})))$ errors \cite[Thm.~2]{yarotsky2018}. This result matches the lower bound based on VC dimension of Anthony and Barlett \cite[Thm.~8.7]{anthony2009} (see also \cite{harvey2017}), and improves the $\mathcal{O}(W^{-1/n}\log_2^{1/n}W)$ errors he obtained for Lipschitz functions in 2017 \cite[Thm.~1]{yarotsky2017}. 

Let us stress that Yarotsky's theorems provide upper bounds for the errors when the same network architecture is used to approximate all functions in a given function space. In other words, the network architecture does not depend on the function being approximated in that space; only the weights do. {Moreover, the networks he utilizes are said to be \textit{very deep} because the depth $L$ satisfies $L=\OO(W)$.} We recall his 2018 result below.

%%%%%%%%%%%%%%%%%%%%%%%%%%%%%%%%%%%%%%%%%%%%%%%%%%%%%%%%%%%%%%%%%%%%%%%%%%%
\begin{theorem}[Approximation of continuous functions by very deep ReLU networks in the unit hypercube]\label{thm:yarotsky}
For any continuous function $f:[0,1]^n\rightarrow\R$ with modulus of continuity $\omega_f$,
there is a deep ReLU network $\wff$ depth $L\leq c_0(n)W$ and size $W$, such that
\begin{align*}
\Vert f - \wff\Vert_{L^\infty([0,1]^n)}\leq c_1(n)\omega_f\left(c_2(n)W^{-2/n}\right),
\end{align*}
for some $c_0(n),c_1(n),c_2(n)>0$.
\end{theorem} 

We extend Yarotsky's result to domains $[0,M]^n$.

%%%%%%%%%%%%%%%%%%%%%%%%%%%%%%%%%%%%%%%%%%%%%%%%%%%%%%%%%%%%%%%%%%%%%%%%%%%
\begin{corollary}[Approximation of continuous functions by very deep ReLU networks in scaled hypercubes]\label{cor:yarotsky}
For any continuous function $f:[0,M]^n\rightarrow\R$ with modulus of continuity $\omega_f$,
there is a deep ReLU network $\wff$ of depth $L\leq c_0(n)W$ and size $W$, such that
\begin{align*}
\Vert f - \wff\Vert_{L^\infty([0,M]^n)}\leq c_1(n)\omega_f\left(c_2(n)M W^{-2/n}\right),
\end{align*}
with $c_0(n),c_1(n),c_2(n)$ as in Theorem~\ref{thm:yarotsky}.
\end{corollary} 

\begin{proof}
We use Theorem~\ref{thm:yarotsky} with $g(x)=f(x/M)$ on $[0,1]^n$. Note that $\omega_g(\delta)=\omega_f(M\delta)$. Therefore, there is a deep ReLU network $\wgg$ of depth $L\leq c_0(n)W$ and size $W$, such that
\begin{align*}
\Vert g - \wgg\Vert_{L^\infty([0,1]^n)}
& \leq c_1(n)\omega_g\left(c_2(n)W^{-2/n}\right), \\
& = c_1(n)\omega_f\left(c_2(n)M W^{-2/n}\right),
\end{align*}
with $c_0(n),c_1(n),c_2(n)$ as in Theorem~\ref{thm:yarotsky}. Since $g(Mx) - \wgg(Mx) = f(x) - \wgg(Mx)$, the network $\wff(x)=\wgg(Mx)$ satisfies all requirements in this corollary.
\end{proof}

We shall now apply Corollary~\ref{cor:yarotsky} to the inner and outer functions of Propositions~\ref{prop:innercon} and \ref{prop:outercon}. For simplicity, we shall assume, throughout the rest of the paper, that $m=2n$ and $\gamma=2n+2$.

%%%%%%%%%%%%%%%%%%%%%%%%%%%%%%%%%%%%%%%%%%%%%%%%%%%%%%%%%%%%%%%%%%%%%%%%%%%
\begin{proposition}[Approximation of the inner function by very deep ReLU networks]\label{prop:innerapprox}
Let $n\geq 2$ be an integer and $\psi$ be the inner function defined in Proposition~\ref{prop:innercon}.
Then, for any scalar $0<\epsilon<1$, there is a deep ReLU network $\wpsi$ that has depth $L\leq c_0(1)W$ and size
\begin{align*}
W\leq c_3(n)\epsilon^{-[1+\log_2(n+1)]/2},
\end{align*}
such that $\Vert\psi - \wpsi\Vert_{L^\infty([0,2])}\leq\epsilon$, with
\begin{align}\label{eq:c_3}
c_3(n) = \left[(2n+5)c_1(1)\right]^{[1+\log_2(n+1)]/2}c_2(1)^{1/2},
\end{align}
and $c_0(1),c_1(1),c_2(1)$ as in Theorem~\ref{thm:yarotsky}.
\end{proposition} 

\begin{proof}
We use Corollary~\ref{cor:yarotsky} with $M=2$ and the modulus of continuity of Proposition~\ref{prop:innerhold}, \textit{i.e.},
\begin{align*}
\omega_\psi(\delta) = \nu\delta^\alpha,
\end{align*}
with $\nu=2^{-\alpha}(2n+5)$ and $\alpha=\log_{2n+2}2$.
\end{proof}

%%%%%%%%%%%%%%%%%%%%%%%%%%%%%%%%%%%%%%%%%%%%%%%%%%%%%%%%%%%%%%%%%%%%%%%%%%%
\begin{proposition}[Approximation of the outer functions by very deep ReLU networks]\label{prop:outerapprox}
Let $n\geq 2$ be an integer, $f:[0,1]^n\rightarrow\R$ be a continuous function in $ K_C([0,1]^n;\,\R)$ that satisfies $\Vert f\Vert_{L^\infty([0,1]^n)}\leq1$, and $\phi_j^r$ be the $(2n+1)$ outer functions defined in Proposition~\ref{prop:outercon} at iteration $r$, for some $r\geq1$. Then, for any scalar $0<\epsilon<1$, there are $(2n+1)$ deep ReLU networks $\wphi_j^r$ that have depth $L\leq c_0(1)W$ and size
\begin{align*}
W\leq c_4(n,r)\epsilon^{-1/2},
\end{align*}
such that $\Vert\phi_j^r - \wphi_j^r\Vert_{L^\infty([0,2\frac{\gamma-1}{\gamma-2}])}\leq\epsilon$, with
\begin{align}\label{eq:c_4}
c_4(n,r) = \left[\frac{c_1(1)c_2(1)}{n}r (2n+2)^{2n^{C}}\right]^{1/2},
\end{align}
and $c_0(1),c_1(1),c_2(1)$ as in Theorem~\ref{thm:yarotsky}.
\end{proposition} 

\begin{proof}
We use Corollary~\ref{cor:yarotsky} with $M=2\frac{\gamma-1}{\gamma-2}$ and the modulus of continuity corresponding to the Lipschitz continuity described in Proposition~\ref{prop:outerlip}, \textit{i.e.},
\begin{align*}
\omega_{\phi_j^r}(\delta) = \nu_r(f)\delta,
\end{align*}
with $\nu_r(f)$ as in Equation~\eqref{eq:nu_r_K_C}. This yields
\begin{align*}
\Vert\phi_j^r - \wphi_j^r\Vert_{L^\infty([0,2\frac{\gamma-1}{\gamma-2}])} & \leq 2c_1(1)c_2(1)\nu_r(f)\frac{\gamma-1}{\gamma-2}W^{-2}, \\
& \leq \frac{c_1(1)c_2(1)}{n}r(2n+2)^{2n^{C}}W^{-2},
\end{align*}
where $\wphi_j^r$ is a very deep neural network with size $W$ and depth $L\leq c_0(1)W$, and $c_0(1),c_1(1),c_2(1)$ as in Theorem~\ref{thm:yarotsky}.
To achieve the $\epsilon$ approximation error, $W$ can be as small as $c_4(n,r) \epsilon^{-1/2}$, where
\begin{align*}
c_4(n,r) = \left[\frac{c_1(1)c_2(1)}{n}r (2n+2)^{2n^{Cr}} \right]^{1/2}.
\end{align*}
\end{proof}

%%%%%%%%%%%%%%%%%%%%%%%%%%%%%%%%%%%%%%%%%%%%%%%%%%%%%%%%%%%%%%%%%%%%%%%%%%%
\section{Main theorem}\label{sec:theorem}

We present in this section our main theorem about the approximation of multivariate continuous functions by very deep ReLU networks. Our proof is based on the Kolmogorov--Arnold superposition theorem (Theorem~\ref{thm:kast}), and on the approximation of the inner and outer functions by very deep ReLU networks (Propositions~\ref{prop:innerapprox} and \ref{prop:outerapprox}).

%%%%%%%%%%%%%%%%%%%%%%%%%%%%%%%%%%%%%%%%%%%%%%%%%%%%%%%%%%%%%%%%%%%%%%%%%%%
\begin{theorem}[Approximation of continuous functions using the Kolmogorov--Arnold superposition theorem]\label{thm:deepkast}
Let $n\geq 2$ be an integer and $f$ be a continuous function in $ K_C([0,1]^n;\,\R)$ that satisfies $\Vert f\Vert_{L^\infty([0,1]^n)}\leq1$. Then, for any scalar $0<\epsilon<1$, there is a deep ReLU network $\wff_r$ that has depth
\begin{align*}
L\leq & c_0(1)\tilde{c}_3(n,r(\epsilon))\epsilon^{-[1+\log_2(n+1)]/2} \\
& + c_0(1)\tilde{c}_4(n,r(\epsilon))\epsilon^{-1/2},
\end{align*}
and size
\begin{align*}
W\leq & n(2n+1)\tilde{c}_3(n,r(\epsilon))\epsilon^{-[1+\log_2(n+1)]/2} \\
& + (2n+1)\tilde{c}_4(n,r(\epsilon))\epsilon^{-1/2},
\end{align*}
such that $\Vert f - \wff_r\Vert_{L^\infty([0,1]^n)}\leq\epsilon$, with $c_0(1)$ as in Theorem~\ref{thm:yarotsky},
\begin{align*}
\tilde{c}_3(n,r(\epsilon)) & \hspace{-0.055cm}=\hspace{-0.055cm}\left[\frac{4n+2}{n}\hspace{-0.05cm}r(\epsilon)(2n+2)^{2n^C}\right]^{[1+\log_2(n+1)]/2}\hspace{-1cm}c_3(n), \\
\tilde{c}_4(n,r(\epsilon)) & \hspace{-0.055cm}=\hspace{-0.055cm}\left[8n+4\right]^{1/2}c_4(n,r(\epsilon)),
\end{align*}
$c_3(n)$ as in Equation~\eqref{eq:c_3}, $c_4(n,r)$ as in Equation~\eqref{eq:c_4}, and $r(\epsilon)=\ceil{\log 2\epsilon^{-1}/\log\eta^{-1}}$.
\end{theorem}

\begin{proof}
Let $0<\epsilon<1$ be a scalar. Let $f:[0,1]^n\rightarrow\R$ be a continuous function in $ K_C([0,1]^n;\,\R)$ that satisfies $\Vert f\Vert_{L^\infty([0,1]^n)}\leq1$. 
Using Equation~\eqref{eq:kast} in Theorem~\ref{thm:kast}, we write $f$ as
\begin{align*}
f(x_1,\ldots,x_n) = \sum_{j=0}^{2n}\phi_j\left(\sum_{i=1}^n\lambda_i\psi(x_i+ja)\right).
\end{align*}
We first approximate $f$ by $f_r$ defined in Equation~\eqref{eq:f_r} using the error bound in Proposition~\ref{prop:outerapprox}, i.e.,
\begin{align*}
f_r(x_1,\ldots,x_n) = \sum_{j=0}^{2n}\phi_j^r\left(\sum_{i=1}^n\lambda_i\psi(x_i+ja)\right).
\end{align*}
If we choose $r(\epsilon)=\ceil{\log 2\epsilon^{-1}/\log\eta^{-1}}$, then using Equation~\eqref{eq:e_r}, we get $\Vert f - f_r\Vert_{L^\infty([0,1]^n)}\leq\epsilon/2$.

We now approximate $f_r$ by a deep ReLU network $\wff_r$ defined by
\begin{align}\label{eq:network}
\wff_r(x_1,\ldots,x_n) = \sum_{j=0}^{2n}\wphi_j^r\left(\sum_{i=1}^n\lambda_i\wpsi(x_i+ja)\right),
\end{align}
where $\wpsi$ and $\wphi_j^r$ approximate $\psi$ and $\phi_j^r$ to some accuracies $0<\epsilon_{\psi}<1$ and $0<\epsilon_{\phi}<1$ to be determined later. We plot the subnetwork $\wphi_j^r$ in Figure~\ref{fig:network}.

Using Propositions~\ref{prop:innerapprox} and \ref{prop:outerapprox}, the network $\wpsi$ has depth $L_{\psi}\leq c_0(n)W_{\psi}$ and size
\begin{align*}
W_{\psi}\leq c_3(n)\epsilon_{\psi}^{-[1+\log_2(n+1)]/2},
\end{align*}
while the networks $\wphi_j^r$ have depth $L_{\phi}\leq c_0(n)W_\phi$ and size
\begin{align*}
W_{\phi}\leq c_4(n,r)\epsilon_{\phi}^{-1/2}.
\end{align*}

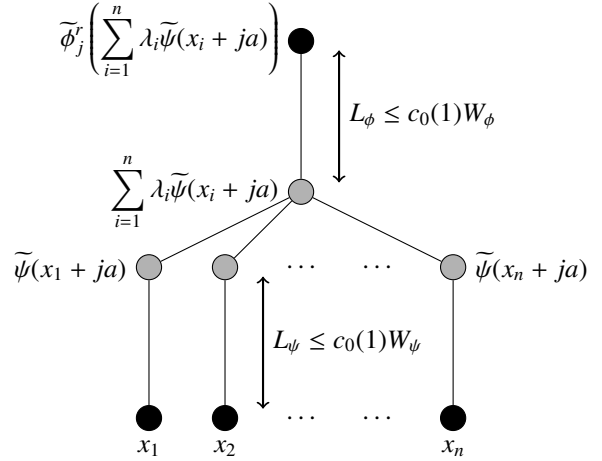
\begin{figure}[t]
\centering
\begin{tikzpicture}
[   cnode/.style={draw=black, fill=#1, minimum width=3mm, circle},
]

	% Inputs:
    \node[cnode = black, label = 270: $x_1$] (x_1) at (0, -3) {};
    \node[cnode = black, label = 270: $x_2$] (x_2) at (1, -3) {};
    \node at (2, -3) {$\ldots$};
    \node at (3, -3) {$\ldots$};
    \node[cnode = black, label = 270: $x_n$] (x_n) at (4, -3) {};
   
    % Inner function:
    \node[cnode = gray, label = 180: $\wpsi(x_1+ja)$] (psi_1) at (0, -1) {};
    \node[cnode = gray] (psi_2) at (1, -1) {};
    \node at (2, -1) {$\ldots$};
    \node at (3, -1) {$\ldots$};
    \node[cnode = gray, label = 360: $\wpsi(x_n+ja)$] (psi_n) at (4, -1) {};
    \draw (x_1) -- (psi_1) {};
    \draw (x_2) -- (psi_2) {};
    \draw (x_n) -- (psi_n) {};
    
    % Outer function:
    \node[cnode = gray, label = 180: $\dsp\sum_{i=1}^n\lambda_i\wpsi(x_i+ja)$] (psi_sum) at (2, 0) {};
    \draw (psi_1) -- (psi_sum) {};
    \draw (psi_2) -- (psi_sum) {};
    \draw (psi_n) -- (psi_sum) {};
    \node[cnode = black, label = 180: $\dsp\wphi_j^r\left(\sum_{i=1}^n\lambda_i\wpsi(x_i+ja)\right)$] (phi) at (2, 2) {};
    \draw (psi_sum) -- (phi) {};
    
    % Depth of the network:
    \node[draw=none] (ghost1) at (1.5, -3) {};
    \node[draw=none] (ghost2) at (1.5, -1) {};
    \draw [thick, <->] (ghost1) -- node[right]{$L_\psi\leq c_0(1)W_\psi$} (ghost2);
    \node[draw=none] (ghost3) at (2.5, 0) {};
    \node[draw=none] (ghost4) at (2.5, 2) {};
    \draw [thick, <->] (ghost3) -- node[right]{$L_\phi\leq c_0(1)W_\phi$} (ghost4);
    
\end{tikzpicture}
\caption{\textit{Subnetwork $\wphi_j^r$ that approximates the outer function $\phi_j^r$. The deep ReLU network in Equation~\eqref{eq:network} is the sum of $2n+1$ such subnetworks. Each subnetwork has depth $L_\psi+L_\phi$ and size $nW_\psi+W_\phi$, so that the network in Equation~\eqref{eq:network} has depth $L_\psi+L_\phi$ and size $(2n^2+n)W_\psi+(2n+1)W_\phi$.}}
\label{fig:network}
\end{figure}

Using the triangle inequality, we compute the accuracy of the network $\wff_r$ as follows,
\begin{align*}
& \vert f_r(x_1,\ldots,x_n) - \wff_r(x_1,\ldots,x_n)\vert, \\
\leq & \left\vert\sum_{j=0}^{2n}\phi_j^r\left(\sum_{p=1}^n\lambda_i\psi(x_i+ja)\right)-\sum_{j=0}^{2n}\phi_j^r\left(\sum_{i=1}^n\lambda_i\wpsi(x_i+ja)\right)\right\vert \\ 
+ & \left\vert\sum_{j=0}^{2n}\phi_j^r\left(\sum_{i=1}^n\lambda_i\wpsi(x_i+ja)\right)-\sum_{j=0}^{2n}\wphi_j^r\left(\sum_{i=1}^n\lambda_i\wpsi(x_i+ja)\right)\right\vert, \\
\leq & \frac{(2n+1)^2}{2n}\nu_r(f)\epsilon_\psi + (2n+1)\epsilon_{\phi}.
\end{align*}
We must choose
\begin{align*}
\epsilon_\psi = \frac{n\epsilon}{2(2n+1)^2\nu_r(f)}, \quad \epsilon_{\phi} = \frac{\epsilon}{4(2n+1)},
\end{align*}
to obtain $\Vert f_r - \wff_r\Vert_{L^\infty([0,1]^n)}\leq\epsilon/2$ and $\Vert f - \wff_r\Vert_{L^\infty([0,1]^n)}\leq\epsilon$.

Therefore, the network $\wpsi$ has depth $L_{\psi}\leq c_0(n)W_{\psi}$ and size
\begin{align*}
W_{\psi} \leq \tilde{c}_3(n,\epsilon)\epsilon^{-[1+\log_2(n+1)]/2}, 
\end{align*}
with
\begin{align*}
\tilde{c}_3(n,r(\epsilon)) \hspace{-0.055cm}=\hspace{-0.055cm}\left[\frac{4n+2}{n}r(\epsilon)(2n+2)^{2n^C}\right]^{[1+\log_2(n+1)]/2}\hspace{-1cm}c_3(n),
\end{align*}
while the networks $\wphi_j^r$ have depth $L_{\phi}\leq c_0(n)W_\phi$ and size
\begin{align*}
W_{\phi} \leq \tilde{c}_4(n,r(\epsilon))\epsilon^{-1/2}, 
\end{align*}
with
\begin{align*}
\tilde{c}_4(n,r(\epsilon)) & = \left[8n+4\right]^{1/2}c_4(n,r(\epsilon)).
\end{align*}
Lastly, the network $\wff_r$ has depth $L\leq c_0(1)(W_\psi+W_\phi)$ and size $W\leq n(2n+1)W_\psi+(2n+1)W_\phi$.
\end{proof}

The upper bounds in Theorem~\ref{thm:deepkast} show that, for a given dimension $n$, the depth and the size of the network grow like $\OO\left(\epsilon^{-\log n}\right)$; the curse of dimensionality is lessened asymptotically when $\epsilon$ approaches $0$.

{Let us end this section with a comment about smoothness. Yarotsky proved in 2017 that deep ReLU networks of depth and size $\OO(\epsilon^{-n/m})$ can approximate functions with $m$ weak and bounded derivatives in $[0,1]^n$ to accuracy $\epsilon$} \cite[Thm.~1]{yarotsky2017} {(we omitted a logarithmic factor for simplicity). For given $n$ and large enough $m$, $\OO(\epsilon^{-n/m})$ may be smaller than $\OO(\epsilon^{-\log n})$. Conversely, however, for given $m$ and large enough $n$, $\OO(\epsilon^{-n/m})$ may be greater than $\OO(\epsilon^{-\log n})$.

%%%%%%%%%%%%%%%%%%%%%%%%%%%%%%%%%%%%%%%%%%%%%%%%%%%%%%%%%%%%%%%%%%%%%%%%%%%
\section{Discussion}

We have proven upper bounds for the approximation of multivariate functions $f:[0,1]^n\rightarrow\R$ by deep ReLU networks, for which the curse of dimensionality is lessened. The depth and the size of the networks to approximate such functions $f$ grow like $\OO(\epsilon^{-\log n})$, as opposed to $\OO(\epsilon^{-n})$. The proof is based on the ability of very deep ReLU networks to implement the Kolmogorov--Arnold superposition theorem.

There are many ways in which {this work could be fruitfully continued.} If we were able to construct a Lipschitz continuous inner function, we would be able to obtain $\OO(\epsilon^{-1})$ estimates. Actor and Knepley designed in 2017 an algorithm to compute a Lipschitz continuous inner function, but they did not provide a method to compute the outer functions \cite{actor2017}.

From a theoretical point of view, it would be interesting to investigate error bounds for deep networks with other activation functions (\textit{e.g.}, sigmoid and tanh). Some results about approximation by deep networks using smooth activation functions can be found in} \cite{poggio2017} and the references therein. Using these results, it would be possible to derive an analogue of Theorem~\ref{thm:deepkast}. Let us highlight that, from a numerical point of view, only the ReLU activation function (combined with other tricks) can avoid the gradient degeneracy during network training.

%We conclude this paper by noting that the constants $\tilde{c}_3(n,f)$ and $\tilde{c}_4(n,f)$ in Theorem~\ref{thm:deepkast} depend on $f$ via the integers $k_\ell(f)$, which appear in the construction of the outer functions (Proposition~\ref{prop:outercon}). It would also be interesting future work to derive quantitative estimates for these integers in function of the smoothness of $f$; intuitively, the smoother $f$, the smaller $k_\ell(f)$. Let us emphasize, though, that adding smoothness to the function $f$ is not going to affect the $\OO(\epsilon^{-\log n})$ bound, which comes from the fact that the inner function in the Kolmogorov--Arnold superposition theorem, which does not depend on $f$, is H\"{o}lder continuous.

%%%%%%%%%%%%%%%%%%%%%%%%%%%%%%%%%%%%%%%%%%%%%%%%%%%%%%%%%%%%%%%%%%%%%%%%%%%
\section*{Acknowledgements}

The research of the second author is supported by the start-up grant of the Department of Mathematics at the National University of Singapore and by the Ministry of Education in Singapore under the grant MOE2018-T2-2-147.

%%%%%%%%%%%%%%%%%%%%%%%%%%%%%%%%%%%%%%%%%%%%%%%%%%%%%%%%%%%%%%%%%%%%%%%%%%%
%\bibliography{references.bib}
\bibliography{/Users/hangocquyen/Dropbox/HM/WORK/ARTICLES_BOOKS/BIB_FILES/references.bib}

%%%%%%%%%%%%%%%%%%%%%%%%%%%%%%%%%%%%%%%%%%%%%%%%%%%%%%%%%%%%%%%%%%%%%%%%%%%
\end{document}